\documentclass[a4paper,12pt]{amsart}
\usepackage[a4paper,margin=2.4cm]{geometry}
\usepackage{hyperref}
\usepackage{color}
\usepackage{amssymb}
\usepackage{amsmath}
\usepackage{mathrsfs}
\usepackage{cancel}
\usepackage{xcolor}
\allowdisplaybreaks[4]
\usepackage[dvipsnames]{xcolor}
\usepackage{listings}
\usepackage{xcolor}

\definecolor{codebg}{RGB}{245, 245, 240}       
\definecolor{codezebra}{RGB}{238, 238, 233}    
\definecolor{codecomment}{RGB}{0, 160, 0}      
\definecolor{codekeyword}{RGB}{230, 0, 115}    
\definecolor{codestring}{RGB}{128, 0, 128}     
\definecolor{codenumbers}{RGB}{150, 150, 150}   

\AtEndDocument{\vfill\eject\batchmode}
\newtheorem{Theorem}{Theorem}[section]

\newtheorem{lemma}[Theorem]{Lemma}
\newtheorem{prop}[Theorem]{Proposition}

\newtheorem{defn}{Definition}[section]

\newtheorem{rem}[Theorem]{Remark}

\newtheorem{theoremA}{Theorem}

\newtheorem{corollaryA}{Corollary}

\newcounter{numl}
\newcommand{\labelnuml}{\textup{(\roman{numl})}}

\DeclareSymbolFont{script}{U}{eus}{m}{n}
\DeclareSymbolFontAlphabet{\amathscr}{script}
\DeclareMathSymbol{\Wedge}{0}{script}{"5E}
\DeclareMathAlphabet{\mathrmsl}{OT1}{cmr}{m}{sl}

\newcommand{\Scal}{\mathrm{Scal}}
\renewcommand{\P}{\mathbb{P}}
\newcommand{\T}{\mathbb{T}}

\makeatletter
\newsavebox{\@brx}
\newcommand{\llangle}[1][]{\savebox{\@brx}{\(\m@th{#1\langle}\)}%
  \mathopen{\copy\@brx\kern-0.5\wd\@brx\usebox{\@brx}}}
\newcommand{\rrangle}[1][]{\savebox{\@brx}{\(\m@th{#1\rangle}\)}%
  \mathclose{\copy\@brx\kern-0.5\wd\@brx\usebox{\@brx}}}
\makeatother

\begin{document}

\title[New extremal K\"ahler metrics]{New extremal K\"ahler metrics on projective bundles}
\author{Simon Jubert}

\date{\today}

\begin{abstract}
Consider a holomorphic vector bundle $E$ over a compact complex curve $C$ which decomposes as a sum of stable vector bundles. For the projectivization $\mathbb{P}(E)$, we prove that the existence of a compatible extremal almost K\"ahler (aK) metric of involutive type in the sense of Lejmi is equivalent to the existence of a Calabi extremal K\"ahler metric. This result rests on the Yau--Tian--Donaldson correspondence in terms of the moment polytope $\Delta$ for $\mathbb{P}(E)$, proved by the author and Yin in a previous work. The main advantage is that compatible extremal aK metrics of involutive type are solutions to a second-order linear PDE, rather than a fourth-order nonlinear PDE for Calabi's extremal K\"ahler metrics.

As an application, we prove that when $E$ has rank $4$ and $C$ is an elliptic curve or the projective line, $\mathbb{P}(E)$ is a Calabi dream manifold, i.e. admits an extremal K\"ahler metric in every K\"ahler class.

\end{abstract}

\maketitle

\section{Introduction}

The search for canonical metrics on a given compact K\"ahler manifold $(M,J,\omega)$ is a central problem in K\"ahler geometry. In the 1980s, Calabi \cite{Cal82} proposed the problem of finding an extremal Kähler metric within a fixed Kähler class $[\omega]$. Extremal Kähler metrics naturally generalize constant scalar curvature Kähler (cscK) metrics and are defined by the property that the hamiltonian vector field generated by their scalar curvature is holomorphic.

Following several decades of effort (see e.g. \cite{CC21c, CDS12, DR, Tia91, Yau78}), it is now established \cite{BJ25, DZ25, Li22, Tru26} that the existence of an extremal K\"ahler metric can be characterized by an algebraic condition called \textit{relative uniform K-stability} of $(M,J,[\omega])$. This result is known as the Yau--Tian--Donaldson (YTD) correspondence.

The algebraic objects involved in the definition of relative uniform K-stability are called \textit{test configurations} or, more generally, \textit{models}, which are certain degenerations of the manifold over the projective line. By definition, relative uniform K-stability requires a certain positivity condition when evaluating the \textit{Donaldson--Futaki invariant} on these degenerations.

For a general Kähler manifold, the notion of relative uniform K-stability is sophisticated and remains hard to test on explicit examples. In particular, the problem of determining whether an extremal metric exists in a given Kähler class remains widely open.

\smallskip

Several strategies have been developed to test stability and thereby understand the existence of an extremal K\"ahler metric. One active line of inquiry is the use of the delta invariant, defined as a $K$-stability threshold \cite{BJ20, FO18}. Several equivalent definitions have been provided, originating from both analytic and algebraic perspectives; see e.g. \cite{ABS25, DJ25, RTZ21, Z21}.

In this work, we use a different approach based on a \textit{toric symplectic viewpoint} and a suitable notion of stability for \textit{moment polytopes}, following \cite{ACGT06, ACGT11, Don02}.

\smallskip

On a certain class of Kähler manifolds with a large symmetry group, the relative uniform K-stability of a Kähler class $[\omega]$, restricted to a specific subclass of test configurations, translates into a stability condition on a moment polytope $\Delta$. The fundamental idea was introduced by Donaldson in the case of toric manifolds \cite{Don02}, where he defined \textit{toric test configurations} as degenerations associated with certain piecewise-linear convex functions on $\Delta$. The works of Chen--Cheng and He \cite{CC21a, He19}, building on several previous efforts (see e.g. \cite{Ab08, Gui, ZZ}), established that testing stability solely on toric test configurations suffices to guarantee the existence of an extremal K\"ahler metric in the corresponding Kähler class. Several instances of this phenomenon have been discovered, notably for spherical manifolds by Delcroix \cite{Del23}, various toric fibrations \cite{ACGT08, ACGT11, Jub23, LLS21}, and, most importantly for our purposes, projective bundles over curves by the author and Yin \cite{JY26}.

The idea of this work is to relate the stability of the moment polytope $\Delta$ to the existence of a certain class of canonical metrics on an almost K\"ahler manifold, whose existence can be effectively verified in practice.

\smallskip

Recall that an almost Kähler manifold $(M,J,\omega)$ is a symplectic manifold $(M,\omega)$ endowed with a compatible almost complex structure $J$. In \cite{Lej10}, Lejmi introduced the notion of an \textit{extremal almost Kähler metric} on a symplectic manifold $(M,\omega)$. More precisely, a compatible almost complex structure $J$ on $(M,\omega)$ induces a hermitian connection $\nabla$. From the connection $\nabla$, one can define the hermitian scalar curvature $\mathrm{Scal}^{\nabla}(J)$ associated with the almost Kähler structure $(M,J,\omega)$.
Hence, $(M,J,\omega)$ is said to be extremal if the hamiltonian vector field of its hermitian scalar curvature preserves $J$.

\smallskip
We now describe the extremal almost K\"ahler equation for the manifolds of interest in this paper. We consider the projectivization $\mathbb{P}(E)$ of a holomorphic vector bundle $E$ over a  smooth compact complex curve $C$. By \cite[Theorem 2]{AK19}, a necessary condition for the existence of an extremal metric is that $E$ splits as a direct sum of $\ell+1$ stable (or, equivalently, projectively flat) vector bundles. Throughout this paper, we therefore assume without loss of generality that $E$ is of this form. This splitting induces an action of the torus $\mathbb{T} \cong (S^1)^\ell$ on $\mathbb{P}(E)$. We fix a Kähler class $[\tilde{\omega}_0]$ on $\mathbb{P}(E)$ and, without loss of generality, assume that $[\tilde{\omega}_0]$ restricts to the Kähler class of the Fubini--Study metric on each fiber. Then, the moment polytope $\Delta$ associated with $(\mathbb{P}(E),[\tilde{\omega}_0],\mathbb{T})$ is the standard $\ell$-simplex in $\mathbb{R}^\ell$.

In \cite{ACGT11}, the authors introduced the notion of compatible almost Kähler metrics of involutive type on the projective bundle $\mathbb{P}(E)$. By definition, these are almost Kähler metrics induced by smooth matrix-valued functions $\mathbf{H}$ on $\Delta$. This notion was first introduced for toric manifolds \cite{Don02, Lej10}. The associated almost K\"ahler metric is then extremal if and only if $\mathbf{H}=(H_{ij})_{i,j=1}^\ell$ satisfies the second-order linear PDE
\begin{equation}{\label{i:ext}}
\sum_{i,j=1}^\ell (H_{ij}v)_{,ij} = vw,
\end{equation}
where $v$ and $w$ are smooth functions on the polytope determined by the topological data of $\mathbb{P}(E)$ such that $vw$ is polynomial, and the index notation $,ij$ denotes second partial derivatives with respect to a fixed basis.

The relationship between relative uniform K-stability and the existence of extremal almost Kähler metrics on a manifold admitting a Kähler structure remains a largely open question. In the case of toric manifolds, Legendre \cite{Leg19} proved that the existence of an extremal almost Kähler metric of involutive type implies the stability of the moment polytope, a result that was later generalized to semisimple principal fibrations by the author \cite{Jub23}.

\smallskip

Our main theorem is the following.

\begin{theoremA}{\label{t:A}}
Consider a holomorphic vector bundle $E \to C$ which decomposes as a direct sum of stable vector bundles. If there exists a compatible extremal almost Kähler metric of involutive type on $\mathbb{P}(E)$, then $\Delta$ is $(v,w)$-uniformly K-stable.
\end{theoremA}

The proof generalizes the argument of \cite{CLS14} from the unweighted Kähler setting. Compared with the author's previous generalization to semisimple principal fibrations in \cite{Jub23}, the argument requires a further modification, since in the present setting the weight $v$ may vanish along the boundary of the polytope. The main new technical ingredient is a compactness result with respect to the $L^1$ topology associated with possibly vanishing weighted measures for a certain class of convex functions on $\Delta$.

\smallskip

Donaldson conjectured \cite{Don02} that on toric manifolds the existence of an extremal almost K\"ahler metric is equivalent to the existence of an extremal K\"ahler metric. This conjecture was confirmed in \cite{Leg19} and proved for certain toric fibrations in \cite{Jub23}. As a consequence of the above theorem and the YTD correspondence for projective bundles over a curve \cite{JY26} in term of $\Delta$, we obtain new examples of this phenomenon.

\begin{corollaryA}{\label{c:A}}
Consider a holomorphic vector bundle $E \to C$ which decomposes as a direct sum of stable vector bundles. Then the following statements are equivalent.

\begin{enumerate}
\item There exists an extremal K\"ahler metric $\tilde{\omega}$ in $(\mathbb{P}(E),[\tilde{\omega}_0])$.
\item There exists a compatible involutive extremal aK metric $\tilde{J}$ on $(\mathbb{P}(E),\tilde{\omega}_0)$.
\item $\Delta$ is $(v,w)$-uniformly K-stable.
\end{enumerate}
\end{corollaryA}

We stress that this result reduces the question of the existence of an extremal K\"ahler metric to the study of the second-order linear PDE \eqref{i:ext} on the polytope $\Delta \subset \mathbb{R}^\ell$.

\smallskip
We then apply the above result to prove the existence of new extremal K\"ahler metric when the fiber is $\P^3$. 

Following an idea from \cite[Section 6.3]{ACGT11} in the case where there is no \textit{blow-down} and the fiber is $\P^2$, we consider matrices $\mathbf{H}=(H_{ij})$ such that $v H_{ij}$ is a fifth-order polynomial in the moment coordinates on $\Delta$. With this choice, equation \eqref{i:ext} reduces to a system of algebraic equations, recalling that $vw$ is also a polynomial on $\Delta$.

Computer-assisted computations prove the existence of solutions to this algebraic system in several cases, thereby establishing the existence of new extremal K\"ahler metrics.

\begin{corollaryA}{\label{c:B}}
Consider a holomorphic vector bundle $E$ of rank $4$ over an elliptic curve or the projective line which decomposes as a direct sum of stable vector bundles. Then its projectivization $\mathbb{P}(E)$ is a Calabi dream manifold, i.e. admits an extremal Kähler metric in every Kähler class.
\end{corollaryA}

\medskip

The organization of the paper is as follows:

\begin{itemize}
\item In Section~\ref{s:2}, we review the Yau--Tian--Donaldson correspondence in terms of the moment polytope for $\P(E)$.
\item In Section~\ref{s:3}, we review the notion of an extremal almost Kähler metric.
\item Section~\ref{s:4} is devoted to proving Theorem~\ref{t:A} and Corollary~\ref{c:A}.
\item In Section~\ref{s:5}, we prove Corollary~\ref{c:B}.
\item In Appendix~\ref{a:1}, we provide the Python script used to assist with our computations in Section~\ref{s:5}.
\end{itemize}

\section*{Acknowledgment}

The author would like to thank Vestislav Apostolov and Eveline Legendre for valuable advice. He also thank Mehdi Lejmi for his interest in this work. The author is funded by the ERC SiGMA - 101125012 (PI: Eleonora
Di Nezza). 

\section{The Calabi problem on projective bundles over a curve}{\label{s:2}}

In this section, we briefly review the notion of weighted stability for moment polytopes and recall the YTD correspondence in term of the moment polytope established in \cite{JY26} for projective bundles over a curve.

\subsection{Stability of polytopes}
Consider a commutative Lie algebra $\mathfrak{t}$ of dimension $\ell$ with a fixed lattice $\Lambda \subset \mathfrak{t}$. We denote by $\mathfrak{t}^*$ the dual of the Lie algebra $\mathfrak{t}$. Let $(\Delta, \mathbf{L})$ be a labeled compact convex polytope, i.e. $\mathbf{L}$ is a collection of affine-linear functions defining $\Delta \subset \mathfrak{t}^*$. A polytope is said to be Delzant if $\Delta$ is compact; $\Delta$ is \textit{simple}, meaning that each vertex of $\Delta$ annihilates precisely $\ell$ of the affine functions in $\mathbf{L}$ and the corresponding normals form a basis of $\mathfrak{t}^*$; $\Delta$ is \textit{integral}, meaning that for each vertex of $\Delta$, the inward normal vectors of the adjacent facets span $\Lambda$.

Following \cite{Don02, JY26, Lah19}, for $v \in \mathcal{C}^\infty(\Delta, \mathbb{R}_{\ge 0})$, $v > 0$ on $\mathring{\Delta}$, and $w \in \mathcal{C}^\infty(\Delta, \mathbb{R})$, we define the $(v,w)$-Donaldson--Futaki invariant of the labeled polytope $(\Delta, \mathbf{L})$ as the functional $\mathcal{F}_{v,w} \colon \mathcal{C}^0(\Delta) \to \mathbb{R}$ given by
\begin{equation}\label{eq:DF}
\mathcal{F}_{v,w}(f) := 2 \int_{\partial \Delta} f v d\sigma - \int_\Delta f w v dx.
\end{equation}

\noindent where $d\sigma$ is the induced measure on each face $F_i \subset \partial \Delta$, defined by $dL_i \wedge d\sigma = -dx$, and $dx$ is the Lebesgue measure on $\Delta$. We suppose that the weights $v$ and $w \in \mathcal{C}^\infty(\Delta, \mathbb{R})$ satisfy $\mathcal{F}_{v,w}(f) = 0$ for all $f$ affine-linear on $\Delta$. Integration by parts (see e.g.~[41]) reveals that the latter condition is necessary for the existence of a $(v,w)$-cscK metric \cite{Lah19} on the toric manifolds corresponding to $(\Delta,\mathbf{L})$ via the Delzant correspondence \cite{Delz}.

Let $\mathcal{CV}^\infty(\Delta)$ be the space of continuous convex functions on $\Delta$, smooth in the interior $\mathring{\Delta}$. We fix $x_0 \in \mathring{\Delta}$. For every $f : \Delta \rightarrow \mathbb{R}$, we denote by $f^*$ the normalization of $f$ such that $f \geq f^*(x_0)=0$.

\begin{defn}
A Delzant polytope $(\Delta,\mathbf{L})$ is said to be $(v,w)$-uniformly K-stable if there exists $\lambda>0$ such that

\begin{equation*}
\mathcal{F}_{v,w}(f) \geq  \lambda \| f^* v \|_{L^1(\Delta)},
\end{equation*}
for all $f \in \mathcal{CV}^\infty(\Delta)$.
\end{defn}

\subsection{YTD for projective bundles}{\label{s:ytd}}
In this section, we recall the YTD correspondence in terms of a moment polytope  for projective bundles from \cite{JY26} and explain geometrically where the associated notion of weighted stability comes from. We begin by recalling the construction of compatible K\"ahler metrics \cite{ACGT04}, as it will clarify which polytope needs to be considered.

Consider a holomorphic vector bundle $E$ over a smooth compact complex curve $C$. Fix a cscK K\"ahler metric $\omega_C$ on $C$ whose K\"ahler class is integral, i.e. $\frac{1}{2\pi}[\omega_C] \in H^2(C,\mathbb{Z})$, and primitive. We consider the projectivization of $E$,

\begin{equation*}
    \mathbb{P}(E) \to C,
\end{equation*}
and we denote by $J_0$ the \textit{standard} complex structure on $\P(E)$. Apostolov--Keller proved \cite[Theorem 2]{AK19} that a necessary condition for a projective bundle $(\P(E),J_0)$ to admit a uniformly K-stable K\"ahler class is that $E$ decomposes as a direct sum 

\begin{equation}{\label{proj:flat}}
    E = \bigoplus_{j=0}^\ell E_j
\end{equation}
of stable (or equivalently projectively flat) vector bundles $E$. Hence, without loss of generality, we assume throughout the paper that $E$ is of this form and we fix $h_j$ a projectively flat hermitian metric on $E_j$.  

We recall that, by definition, a holomorphic vector bundle $E$ of rank $\mathrm{rk}(E)$ over a smooth compact complex curve $C$ is  \textit{Mumford–Takemoto stable} (stable for short) if, for any nonzero proper holomorphic subvector bundle $F \subset E$,

\begin{equation*}
\mu(F) < \mu(E),
\end{equation*}
where $\mu(E)$ denotes the slope of $E$, defined by

\begin{equation*}
\mu(E):= \frac{\int_C c_1(E)}{\mathrm{rk}(E)},
\end{equation*}
$\mu(F)$ denotes the slope of $F$, defined in the same way. In the above definition $c_1(E)\in H^2(C,\mathbb{R})$ is the first Chern class of $E$,

The  splitting \eqref{proj:flat} provides us with a fiberwise diagonal action of a compact torus $(S^1)^{\ell+1}$, where each factor $S^1$ acts fiberwise and diagonally on each vector bundle $E_j$. Hence, passing to the quotient, we obtain a fiberwise action of an $\ell$-dimensional torus $\mathbb{T}$ on $\mathbb{P}(E)$.

The $\mathbb{T}$-equivariant blow-up $\widehat{\mathbb{P}(E)}$ of $\mathbb{P}(E)$ along the submanifolds $\bigcup_{j=0}^\ell \mathbb{P}(E_j)$ is $\mathbb{T}$-equivariantly biholomorphic to

\begin{equation}\label{bu}
\widehat{\mathbb{P}(E)} \cong \mathbb{P}\left(\bigoplus_{j=0}^\ell \mathcal{O}_{\mathbb{P}(E_j)}(-1)\right) \longrightarrow \hat{C},
\end{equation}
where $\hat{C} := \mathbb{P}(E_0) \times_C \dots \times_C \mathbb{P}(E_\ell) \rightarrow C$ is the fiber product over $C$. In particular, $\widehat{\mathbb{P}(E)}$ is a fiber bundle with fiber $\mathbb{P}^\ell:=\P(\mathbb{C}^{\ell+1})$.

Let $P \to \hat{C}$ be the principal $\mathbb{T}$-bundle such that each $S^1$-component of $P$ is the unit $S^1$-bundle inside $(\mathcal{O}_{\mathbb{P}(E_j)}(-1), h_j)$, where, by abuse of notation, $h_j$ also denotes the induced hermitian metric on $\mathcal{O}_{\mathbb{P}(E_j)}(-1)$. Let $\theta \in \Omega^1(P, \mathfrak{t})$ be the associated connection form on $P$, where $\mathfrak{t}$ denotes the Lie algebra of $\T$. In particular, the blow-up $\widehat{\mathbb{P}(E)}$ of $\mathbb{P}(E)$ is related to $P$ by

\begin{equation*}
\widehat{\mathbb{P}(E)} = P \times_\mathbb{T} \mathbb{P}^\ell,
\end{equation*}
where we consider the standard $\mathbb{T}$-action on $\mathbb{P}^\ell$.

Let $\mathfrak{t}^*$ be the dual Lie algebra of $\T$. Let $\omega$ be a K\"ahler metric in the first Chern class $2\pi c_1(\P^\ell)$ of $\P^\ell$ and denote by $\mu_\omega : \P^\ell \rightarrow \mathfrak{t}^*$ its moment map. Let $\Delta := \mu_{\omega}(\mathbb{P}^{\ell})$ be the standard $\ell$-simplex in $\mathfrak{t}^*$ and let $\mathbf{L}=({L_j})_{j=1}^\ell$ be the associated label. In moment-angle coordinates $(\mu_{\omega}, t)$ on the subspace of regular orbits for the $\mathbb{T}$-action $\mathring{\mathbb{P}}^\ell \subset \mathbb{P}^\ell$, the Kähler metric is written as

\begin{equation*}
\omega = \langle d\mu_\omega \wedge dt \rangle,
\end{equation*}
where $\langle \cdot, \cdot \rangle$ denotes the natural pairing between $\mathfrak{t}^*$ and $\mathfrak{t}$, and $t \colon \mathring{\mathbb{P}}^\ell \to \mathfrak{t}$ is the angular coordinate. Now consider the subspace of regular orbits $\mathring{\widehat{\P(E)}}$ for the $\mathbb{T}$-action on $\widehat{\mathbb{P}(E)}$. We can consider, as in \cite[(6)]{ACGT11}, the K\"ahler metric on (the non-compact manifold) $\mathring{\widehat{\mathbb{P}(E)}}$,

\begin{equation}{\label{comp:met}}
\tilde{\omega} = \sum_{\substack{j=0 \\ d_j \neq 1}}^\ell  L_j(\mu_\omega) \omega_{\mathrm{FS},j} + \left(c - \sum_{j=0}^\ell \mu(E_j)L_j(\mu_\omega)\right)\omega_C + \langle d\mu_\omega \wedge \theta \rangle,
\end{equation}
where $\omega_{\mathrm{FS},j}$ is the Fubini-Study metric on $\mathbb{P}(\mathbb{C}^{d_j})$ and $c \in \mathbb{R}$ is chosen such that the affine function $c - \sum_{j=0}^\ell \mu(E_j)L_j(x) > 0$. By abuse of notation, we omit several pullbacks in the above formula. This metric compactifies to a degenerate K\"ahler metric on $\widehat{\mathbb{P}(E)}$ (in the sense that it could vanish along the exceptional divisor of $\eqref{eq:blowdown}$ below), and it is proved in \cite[Theorem 2]{ACGT04} that it is the pullback of a K\"ahler form on $\mathbb{P}(E)$ via the blow-down map

\begin{equation}\label{eq:blowdown}
b \colon \widehat{\mathbb{P}(E)} \to \mathbb{P}(E).
\end{equation}
Such metrics are called \emph{compatible K\"ahler metrics}. A K\"ahler class containing a compatible metric is called a \textit{compatible K\"ahler class}. It follows from \cite[Lemma 3.3]{AK19} that, up to scaling, every K\"ahler class is compatible; hence we will assume that we work in such a class throughout the paper. We mention that in the construction of $\tilde{\omega}$, varying the constant $c \in \mathbb{R}$ corresponds to moving the K\"ahler class $[\tilde{\omega}]$ inside $H^2(\mathbb{P}(E), \mathbb{R})$.

Since the blow-up is $\mathbb{T}$-equivariant, the moment map $\mu_{\omega} \colon \mathbb{P}^\ell \to \Delta$ descends to a map

$$
    \mu_{\tilde{\omega}} \colon \mathbb{P}(E) \to \Delta,
$$
and it follows from \eqref{comp:met} and from the construction of a compatible metric that $\mu_{\tilde{\omega}}$ is a moment map for $\tilde{\omega}$.

We summarize this as follows.

\begin{lemma}{\label{l:pol}}\textnormal{(\cite{ACGT04})}.
The moment polytope associated with $(\mathbb{P}(E), [\tilde{\omega}], \mathbb{T})$ is the standard $\ell$-simplex in  $\mathfrak{t}^* \cong \mathbb{R}^\ell$, where $\ell+1$ is the number of irreducible stable components of $E$ in \eqref{proj:flat}.
\end{lemma}

\medskip

We now explain how to translate the notion of relative uniform K-stability for $(\mathbb{P}(E), J_0, [\tilde{\omega}], \mathbb{T})$ in terms of weighted stability of $\Delta$ for suitable weight functions, following \cite{JY26}.

Let $f$ be a convex, piecewise linear function on $\Delta$. Let $(\mathcal{Z}_f, \mathcal{A}_{\mathcal{Z}_f}, \mathbb{T} \times \mathbb{S}^1)$ be the toric test configuration associated with $f$ via Donaldson's construction \cite{Don02} (see \cite{DR17} for the non-polarized case). Without loss of generality, we assume that $(\mathcal{Z}_f, \mathcal{A}_{\mathcal{Z}_f}, \mathbb{T} \times \mathbb{S}^1)$ is smooth. We can then consider
\begin{equation*}
\widehat{\mathcal{Y}}_f := \mathcal{Z}_f \times_{\mathbb{T} \times \mathbb{S}^1} (P \times \mathbb{S}^1).
\end{equation*}
As a consequence of \cite[Propositions 6.3 \& 6.4]{JY26} and the semisimple rigid construction of \cite[Theorem 2]{ACGT04}, $\widehat{\mathcal{Y}}_f$ is the blow-up of a certain fibration
\begin{equation}\label{comp:test}
\mathcal{Y}_f \to C.
\end{equation}
Moreover, following \cite[Propositions 6.3 \& 6.4]{JY26}, we can construct a K\"ahler class $\mathcal{A}_{\mathcal{Y}_f}$ such that $(\mathcal{Y}_f, \mathcal{A}_{\mathcal{Y}_f}, \mathbb{T} \times \mathbb{S}^1)$ is a $\mathbb{T}$-equivariant test configuration for $(\mathbb{P}(E),J_0, [\tilde{\omega}], \mathbb{T})$. According to \cite{JY26}, such test configurations are called \textit{compatible}.

\begin{rem}
In \cite{JY26}, compatible test configurations are obtained using the theory of horospherical test configurations due to Delcroix \cite{Del23}; the viewpoint described above is a consequence of the definition provided in \cite{JY26}.
\end{rem}

We denote by $\mathbf{DF}$ the $\mathbb{T}$-relative Donaldson--Futaki invariant \cite[(2.1)]{Der18} of $(\mathbb{P}(E), J_0, [\tilde{\omega}], \mathbb{T})$. We have the following result from \cite[Lemma 3.12 \& Eq.~(40)]{JY26}, relating $\mathbf{DF}$ and the $(v,w)$-weighted Donaldson--Futaki invariant $\mathcal{F}_{v,w} : \Delta \to \mathbb{R}$ for the weights

\begin{equation}\label{weight} \left\{ \begin{aligned} v(x) & := \prod_{j=0}^\ell L_j(x)^{d_j-1} \left( c - \sum_{j=0}^\ell \mu(E_j)L_j(x) \right), \\[2ex] w(x) & := \ell_{\mathrm{ext}}(x) - \sum_{\substack{j=0 \\ d_j \neq 1}}^\ell \frac{2d_j(d_j - 1)}{L_j(x)} - \frac{4(1 - \mathbf{g})}{c - \sum_{j=0}^\ell \mu(E_j)L_j(x)}, \end{aligned} \right. \end{equation}
where $\mathbf{g}$ is the genus of $C$, $d_j$ is the rank of $E_j$ and $\ell_{\mathrm{ext}}$ is the affine extremal function of $(\mathbb{P}(E),J_0,[\tilde{\omega}_0],\T)$.

\begin{prop}
Let $(\mathcal{Y}_f, \mathcal{A}_{\mathcal{Y}_f}, \mathbb{T} \times \mathbb{S}^1)$ be a compatible test configuration for $(\mathbb{P}(E), J_0, [\tilde{\omega}], \mathbb{T})$. Then the $\mathbb{T}$-relative Donaldson--Futaki invariant is expressed as
\begin{equation*}
\mathbf{DF}(\mathcal{Y}_f, \mathcal{A}_{\mathcal{Y}_f}) = C \mathcal{F}_{v,w}(f),
\end{equation*}
for a certain explicit constant $C > 0$.
\end{prop}

Finally, we conclude this section by recalling the following Yau--Tian--Donaldson correspondence in terms of the moment polytope for projective bundles over a curve, proved in \cite{JY26}.

\begin{Theorem}{\label{t:jy}}
There exists an extremal K\"ahler metric in $(\mathbb{P}(E),J_0,[\tilde{\omega}])$ if and only if $\Delta$ is $(v,w)$-uniformly K-stable for the weights \eqref{weight}.
\end{Theorem}

In \cite{JY26}, the implication \textit{existence implies stability} was established for a polarized K\"ahler class $[\tilde{\omega}]$. As a consequence of Theorem~\ref{c:A}, this correspondence extends to transcendental classes.

\section{Extremal almost K\"ahler metrics on projective bundles}{\label{s:3}}
In this section, we review the definition of extremal almost K\"ahler metrics introduced by Lejmi \cite{Lej10}, we formalize the notion of weighted constant scalar curvature almost K\"ahler (cscaK) metrics, and finally we recall the relation between these two notions in the particular case of projective bundles, following \cite{ACGT11}.

\subsection{Extremal almost K\"ahler metric}

Let $(M,\omega)$ be a $2n$-dimensional compact symplectic manifold. Consider a maximal torus $\mathbb{T}$ in the group of hamiltonian symplectomorphisms $\mathrm{Ham}(M,\omega)$ of $(M,\omega)$. By Atiyah \cite{Ati82} and Guillemin--Sternberg \cite{GS82}, the image of the moment map $\mu_\omega$ of $\omega$ is a compact convex polytope $\Delta$, called the \textit{moment polytope} of $(M,\omega)$.

Let $AK_\omega^\mathbb{T}$ be the space of $\mathbb{T}$-invariant almost complex structures $J$ compatible with $\omega$, in the sense that $\omega$ is $J$-invariant and $g(\cdot,\cdot) := \omega(\cdot,J\cdot)$ is a riemannian metric. In particular, any $J \in AK_\omega^\mathbb{T}$ defines an almost K\"ahler structure on $M$.  
 For any $J \in AK_\omega^\mathbb{T}$, one can consider its \textit{Chern connection}, defined by

\begin{equation*}
\nabla_X Y := D^g_X Y - \frac{1}{2} (D^g_X J)Y
\end{equation*}
for any vector fields $X, Y$ on $M$, where $D^g$ is the Levi-Civita connection of $g$.

The Chern connection $\nabla$ on $TM$ induces a Hermitian connection on the anticanonical bundle $K_J^{-1}(M) = \bigwedge^{n,0}TM$ with curvature $\sqrt{-1}\mathrm{Ric}^\nabla(J)$, where $\mathrm{Ric}^\nabla$ is a closed real $2$-form, called the \emph{hermitian Ricci form}.

For any $J \in AK_\omega^\mathbb{T}$, we define the \textit{hermitian scalar curvature} of $J$ as the symplectic trace of the hermitian Ricci form,
\begin{equation*}
\mathrm{Scal}^\nabla(J) := 2 \Lambda_\omega\big( \mathrm{Ric}^\nabla(J)\big),
\end{equation*}
where for any $2$-form $\beta$, the symplectic trace is defined by $\Lambda_\omega (\beta) \omega^{n}: = n \beta \wedge \omega^{n-1}$. Since $(M,\omega)$ is fixed, in what follows, we will also refer to $J \in AK^\T_\omega$ as an \textit{almost K\"ahler metric}.

Lejmi showed in \cite[Lemma 3.6]{Lej10} that there exists a unique affine function $\ell_{\mathrm{ext}}$ on $\Delta$ such that the symplectic Futaki invariant $\mathbf{F}_{\omega} \colon \mathrm{Aff}(\Delta) \to \mathbb{R}$ vanishes:

\begin{equation*}
\mathbf{F}_\omega(\ell) := \int_M \ell(\mu_\omega) \big(\mathrm{Scal}^\nabla(J) - \ell_{\mathrm{ext}}(\mu_\omega)\big) \omega^n \equiv 0.
\end{equation*}
We call $\ell_{\mathrm{ext}} \in \mathrm{Aff}(\Delta)$ the \textit{extremal affine function} of $AK_\omega^\mathbb{T}$. Recall that $\mathbf{F}_\omega(\ell)$ do not depends on the choice of $J \in AK^\T$ by \cite[Lemma 3.4]{Lej10}.

We now define one of the central notions of this paper.

\begin{defn}[\cite{Lej10}]
An almost K\"ahler metric $J \in AK_\omega^\mathbb{T}$ is called \emph{extremal} if its hermitian scalar curvature $\mathrm{Scal}^\nabla(J)$ is equal to the extremal affine function $\ell_{\mathrm{ext}}$ pulled back by the moment map
\begin{equation*}
\mathrm{Scal}^\nabla(J) = \ell_{\mathrm{ext}}(\mu_\omega).
\end{equation*}
\end{defn}

Now suppose that $(M,J_0,\omega)$ is Kähler. We stress that, by Moser's Lemma \cite{MS98}, the set of $\mathbb{T}$-invariant Kähler metrics in $(J_0,[\omega])$ embeds into the space $AK^\T_{\omega}$ of $\omega$-compatible almost Kähler structures on $M$. In particular, the affine extremal function in the Kähler sense coincides with the one defined in the almost Kähler category.

Note that in some cases, it can be useful to consider a torus $\T$ that is not maximal in $\mathrm{Ham}(M,\omega)$, while still being able to define the extremal affine function. We will apply this fact to projective bundles in Section \ref{s:compak}.

\subsection{Weighted cscaK metric}

We now generalize the discussion of the previous paragraph. We consider two functions $v, w$ on $\Delta$ such that $v > 0$ on $\mathring{\Delta}$ and $v\geq 0$. We then introduce the weighted hermitian Ricci form
\begin{equation*}
\mathrm{Ric}_v^\nabla(J) := \mathrm{Ric}^\nabla(J) + dd^c \log(v(\mu_\omega)),
\end{equation*}
where for any smooth function $f$, the twisted differential is defined by $d^c:=-Jdf$. Similarly to the above discussion, we define the weighted hermitian scalar curvature by
\begin{equation*}
\mathrm{Scal}_{v}^\nabla(J) := 2 \Lambda_{\omega,v}(\mathrm{Ric}_v^\nabla(J)),
\end{equation*}
where for any $\mathbb{T}$-invariant $J$-invariant $2$-form $\beta$ with moment map $\mu_\beta$, the weighted symplectic trace is given by $\Lambda_{\omega, v} (\beta) := \Lambda_\omega(\beta) + \langle d \log(v(\mu_\omega)), \mu_\beta \rangle$.

We then introduce the following definition.

\begin{defn}
A $\mathbb{T}$-invariant almost K\"ahler metric $J$ on a symplectic manifold $(M,\omega)$ is said to be $(v,w)$-weighted constant scalar curvature almost K\"ahler (cscaK) if its weighted hermitian scalar curvature $\mathrm{Scal}_v^\nabla(J)$ satisfies the following geometric PDE

\begin{equation*}
\mathrm{Scal}_v^\nabla(J) = w(\mu_\omega).
\end{equation*}
\end{defn}

This extends the notion of weighted cscK metrics of Lahdili \cite{Lah19} to the almost K\"ahler category.

Extremal almost K\"ahler metrics are examples of weighted cscaK metrics for $v=1$ and $w=\ell_{\mathrm{ext}}$. Our main motivation for formalizing this notion is the case of projective bundles over a curve discussed in Section~\ref{s:compak} below.

\subsection{Almost K\"ahler metric on symplectic toric manifolds}

In this section, we suppose that $(M,\omega)$ is a symplectic toric manifold, i.e. we suppose moreover that $\dim(\mathbb{T}) = \frac{1}{2}\dim(M)=n$. By abuse of notation, let $\mathfrak{t} \subset TM$ also denote the distribution generated by the fundamental vector fields from $\mathfrak{t}$. Following Lejmi \cite{Lej10}, we say that an almost K\"ahler metric $J$ is of \textit{involutive type} if the distribution $J\mathfrak{t} \subset TM$ is involutive, namely if

\begin{equation*}
[J\mathfrak{t}, J\mathfrak{t}] \subset J\mathfrak{t}.
\end{equation*}
Observe that in the K\"ahler case, this condition is automatic since $[J\mathfrak{t}, J\mathfrak{t}] = \{0\}$.

For an almost K\"ahler metric $J$ of involutive type on a toric manifold $(M,\omega,\mathbb{T})$, the Frobenius theorem implies that there exist angular coordinates $t \colon \mathring{M} \to \mathfrak{t}/2\pi \Lambda$ integrating the annihilator of the distribution $J\mathfrak{t}$. Hence, on the open subset of regular orbits $\mathring{M} \subset M$ for the $\mathbb{T}$-action, the metric $g_J=\omega(\cdot,J\cdot)$ is given by

\begin{equation}\label{met:loc}
g_J = \langle d\mu_\omega, \mathbf{G}, d\mu_\omega \rangle + \langle dt, \mathbf{H}, dt \rangle,
\end{equation}
where $\mathbf{H}$ is an $S^2(\mathfrak{t}^*)$-valued function on $\mathring{\Delta}$ and $\mathbf{G} = \mathbf{H}^{-1}$.

The following result, due to \cite[Proposition 1]{ACGT04}, provides a criterion for the almost K\"ahler metric to compactify to $M$:

\begin{prop}{\label{p:bound}}
Let $(M,\omega)$ be a compact toric symplectic manifold with momentum map $\mu_{\omega} \colon M \to \Delta \subset \mathfrak{t}^*$ and let $\mathbf{H}$ be a positive definite $S^2\mathfrak{t}^*$-valued function on $\Delta^0$. Then $\mathbf{H}$ comes from a $\mathbb{T}$-invariant, $\omega$-compatible almost K\"ahler metric $J$ via \eqref{met:loc} if and only if it satisfies the following conditions:
\begin{itemize}
    \item \textnormal{[smoothness]} $\textnormal{\textbf{H}}$ is the restriction to $\mathring{\Delta}$ of a smooth $S^2\mathfrak{t}^{*}$-valued function on $\Delta$;
 \item \textnormal{[boundary values]} for any point $y$ on the codimension one face $F_j \subset \Delta$ with
inward normal $p_j$, we have

\begin{equation}{\label{bounday(condition-equation}}
 \textnormal{\textbf{H}}_y(p_j , \cdot ) = 0 \text{ and } (d\textnormal{\textbf{H}})_y(p_j , p_j ) = 2p_j,
\end{equation}

\noindent where the differential $d\textnormal{\textbf{H}}$ is viewed as a smooth $S^2\mathfrak{t}^*\otimes \mathfrak{t}$-valued function on $P$;
\item \textnormal{[positivity]} for any point y in the interior of a face $F \subseteq \Delta$, $\textnormal{\textbf{H}}_y(\cdot,\cdot)$ is positive definite
when viewed as a smooth function with values in $S^2(\mathfrak{t}/\mathfrak{t}_F )^*$, where $\mathfrak{t}_F \subset \mathfrak{t}$ is the vector subspace spanned by the
inward normals $u_j$ in $\mathfrak{t}$ to the codimension one faces $F$.
\end{itemize}
\end{prop}

Fix a basis $\mathbf{\xi}=(\xi_1,\dots,\xi_n)$ of $\mathfrak{t}^*$ and the associated moment coordinates $x=(x_1,\dots,x_n)$ on $\Delta$. Hence, the weighted scalar curvature is given by 

\begin{equation*}
\Scal_v(J)= \frac{1}{v} \sum_{i,j=1}^n (H_{ij}v)_{,ij},
\end{equation*}
where $(H_{ij})_{i,j=1}^n$ is the matrix $\mathbf{H}$ in the coordinates $x$ and $(v H_{ij})_{,ij}$ stands for the second partial derivatives of $v H_{ij}$ with respect to $x_i, x_j$. 

For $v=1$, this formula is due to \cite{Lej10}. In the integrable K\"ahler setting, it is due to \cite{AM19} and \cite{Lah19}, and the proof extends to our non-integrable context. 

We deduce that a $\mathbb{T}$-invariant almost K\"ahler metric $J$ of involutive type on a symplectic toric manifold associated with a matrix $\mathbf{H}$ via \eqref{met:loc} is $(v,w)$-weighted cscaK if and only if 

\begin{equation}{\label{abreu}}
    \sum_{i,j=1}^n (v H_{ij})_{,ij} = vw.
\end{equation}
\subsection{Compatible aK metric of involutive type on $\mathbb{P}(E)$}{\label{s:compak}}

In this section, we consider the case of the projective bundle $\mathbb{P}(E)$ and use the same notation as in Section~\ref{s:ytd}. 

Fix a compatible K\"ahler metric $\tilde{\omega}$ on $(\mathbb{P}(E),J_0)$. Consider the space of $\mathbb{T}$-invariant almost K\"ahler structures $AK_{\tilde{\omega}}^\mathbb{T}(\mathbb{P}(E))$ on $\mathbb{P}(E)$ compatible with $\tilde{\omega}$. Let $J$ be a $\mathbb{T}$-invariant involutive almost K\"ahler structure on $(\mathbb{P}^\ell, \omega_{\mathrm{FS}})$ induced by a matrix $\mathbf{H}$ via \eqref{met:loc}. As mentioned in \cite[Appendix A]{ACGT11}, $J$ induces an almost K\"ahler structure $(\tilde{J},g_{\tilde{J}})$ on $(\mathring{\mathbb{P}(E)}, \tilde{\omega})$ given by
\begin{equation*}
    g_{\tilde{J}} = \sum_{\substack{j=0 \\ d_j \neq 1}}^\ell L_j(\mu_{\tilde{\omega}}) g_{\mathrm{FS},j} + \left(c - \sum_{j=0}^\ell L_j(\mu_{\tilde{\omega}})\right) g_C + \langle d\mu_{\tilde{\omega}}, \mathbf{G}, d\mu_{\tilde{\omega}}\rangle + \langle \theta , \mathbf{H}, \theta \rangle,
\end{equation*}
where $g_{\mathrm{FS},j}$ is the Fubini-Study riemannian metric on $\mathbb{P}(\mathbb{C}^{d_j+1})$. By an argument similar to the one in Section~\ref{s:ytd}, $(\tilde{J},g_{\tilde{J}})$ compactifies to a well-defined almost K\"ahler structure, still denoted by $(\tilde{J},g_{\tilde{J}})$, on $\P(E)$.

\begin{defn}[\cite{ACGT11}]
An almost K\"ahler metric $\tilde{J}$ on $(\mathbb{P}(E), \tilde{\omega})$ constructed as above is called compatible and of involutive type.
\end{defn}

We have the following key result characterizing the existence of a compatible extremal almost K\"ahler metric of involutive type on $\mathbb{P}(E)$.

\begin{prop}\label{p:extre:proj}
Consider a compatible almost K\"ahler metric $\tilde{J} \in AK_{\tilde{\omega}}^\mathbb{T}(\mathbb{P}(E))$ of involutive type induced by an almost K\"ahler metric $J$ on $\mathbb{P}^\ell$. Then $\tilde{J}$ is an extremal almost K\"ahler metric on $(\mathbb{P}(E), \tilde{\omega})$
if and only if $J$ is $(v,w)$-cscaK on $\mathbb{P}^\ell$ for the weights given in \eqref{weight}.
\end{prop}

We refer to \cite[Appendix A]{ACGT11} for a proof.

\section{Proof of Theorem \ref{t:A} and Corollary \ref{c:A}}{\label{s:4}}

We use the notation of Section~\ref{s:ytd}. Before proving Theorem~\ref{t:A}, we need to establish a compactness result for a class of convex functions on $\Delta$.

Let $\Delta^*$ be the union of the interior $\mathring{\Delta}$ of $\Delta$ and the interiors $\mathring{F}_j$ of its codimension-one faces  $F_j$:
\begin{equation*}
\Delta^* := \mathring{\Delta} \cup \left( \bigcup_j \mathring{F}_j \right).
\end{equation*}

Let $\mathcal{CV}^*(\Delta)$ be the set of convex functions $f$ on $\Delta^*$ such that
\begin{equation}\label{+norm}
\mathcal{F}_{v}^+(f) := 2 \int_{\partial \Delta} f v \, d\sigma + \sum_{\substack{j=0 \\ d_j \neq 1}}^{n} \int_{\Delta} \frac{2 d_j (d_j - 1)}{L_j} f v \, dx < \infty.
\end{equation}
Since $v$ may vanish on the boundary, it is useful to introduce the above functional as a replacement for Donaldson's boundary norm \cite{Don02}. Note that for any  convex functions $f$ on $\Delta^*$, the integral over the boundary is well-defined, even if $f$ is not defined everywhere on the boundary, since the measure $d\sigma$ is supported on the open facets $\mathring{F}_j$. A similar functional was considered in \cite{Del23} in the case of spherical manifolds. In particular,
\begin{equation}\label{l1:fut+}
\| f^* v \|_{L^1(\Delta)} \leq \mathcal{F}_{v}^+(f^*),
\end{equation}
where we recall that $f^*$ denotes the normalization of $f$ such that $f^* \geq f^*(x_0)$ for a fixed reference point $x_0 \in \mathring{\Delta}$, for any $f \in \mathcal{CV}^*(\Delta)$. We refer to \cite[Lemma 6.3]{Del23} for a proof of the above inequality in the setting of polytopes and weights associated with spherical manifolds; the same argument applies in our context.

We have the following key compactness result.

\begin{prop}\label{p:cpt}
Let $(f_k)_{k \geq 0}$ be a sequence of functions $f_k \in \mathcal{CV}^{\infty}(\Delta)$ satisfying 
\begin{equation*}
    \sup_{k \geq 0} \mathcal{F}^+_v(f^*_k) \leq C
\end{equation*}
for some constant $C > 0$. Then, there exists $f^* \in \mathcal{CV}^*(\Delta)$ such that, up to a subsequence,
\begin{equation*}
 \lim_{k \to \infty} \int_\Delta | f^*_k - f^* | v \, dx = 0.
\end{equation*}
\end{prop}

\begin{proof}
By \cite[Proposition 7.2]{Del23} and \eqref{l1:fut+}, there exists a convex function $f^*$ on $\mathring{\Delta}$ such that $f^*_k$ converges locally uniformly to $f^*$ on $\mathring{\Delta}$ and $f^*\geq f^*(x_0)=0$. Although this result is proved for weights associated with spherical manifolds, the same argument applies in our context.

In particular, arguing as in \cite[(5.2.7)]{Don02}, we can extend $f^*$ to a convex function on $\Delta^*$ satisfying
\begin{equation*}
    f^* \leq \liminf_{k \to \infty} f^*_k
\end{equation*}
on any facet $F_j$ of $\Delta$. Since the normalized convex functions $f^*_k$ are non-negative, Fatou's Lemma applies to both the boundary and interior integrals. We thus infer that
\begin{equation*}
    \int_{\partial \Delta} f^* v \, d\sigma \leq \int_{\partial \Delta} \liminf_{k \to \infty} f^*_k  v \, d\sigma \leq \liminf_{k \to \infty} \int_{\partial \Delta} f^*_k v \, d\sigma,
\end{equation*}
and similarly for the interior terms. This yields $\mathcal{F}_v^+(f^*) \leq \liminf_{k \to \infty} \mathcal{F}_v^+(f_k^*) \leq C < \infty$, which proves that
\begin{equation}\label{bd:f}
    \mathcal{F}_v^+(f^*) < \infty,
\end{equation}
i.e. $f \in \mathcal{CV}^*(\Delta)$. 

We now prove the $L^1$-convergence. For toric manifolds without weights, the original idea was provided by Donaldson, see \cite[End of p. 322]{Don02}. A complete detailed proof can be found in \cite[Proposition 5.2.1]{NS21}. In our weighted context with a vanishing weight, some modifications are necessary.

We fix $\epsilon > 0$. For any $\eta \in (0,1)$, we decompose
\begin{equation*}
    \int_\Delta |f^* - f^*_k| v \, dx = \int_{\eta \Delta} |f^* - f^*_k| v \, dx + \int_{\Delta \setminus \eta \Delta} |f^* - f^*_k| v \, dx.
\end{equation*}

The first term on the right-hand side converges to zero by the local uniform convergence of $f^*_k$ to $f^*$. We then focus on the second term.

We decompose $\Delta$ using the cones $C(F_j) := \{ ty + (1-t) x_0 \mid t \in [0,1], \, y \in F_j \}$ generated by each facet $F_j$ with origin $x_0 \in \mathring{\Delta}$. Without loss of generality, assuming $x_0 = 0$, we consider the compact subset $V_j := \{ ty \mid t \in [\eta, 1], \, y \in F_j \}$ inside $C(F_j)$. Hence $\Delta \setminus \eta \Delta = \bigcup_j V_j$. We then decompose the integral
\begin{equation*}
    \int_{\Delta \setminus \eta \Delta} |f^* - f^*_k| v \, dx \le \sum_j \int_{V_j} (f^* + f^*_k) v \, dx,
\end{equation*}
since $f^*$ and $f_k$ are non-negative. 

\smallskip

First, consider a set $V_j$ such that $v = 0$ on the facet $F_j \subset \partial V_j$, which implies that $d_j>1$, see e.g. \cite[Section 3]{ACGT11}. We choose polar coordinates $(t,y)$ associated with the facet $F_j$. Let us write $L_j(x)=\langle p_j, x \rangle + c_j$. In these coordinates, the volume form is $dx = c_j t^{n-1} dt \, d\sigma(y)$, and the affine distance function satisfies $\langle p_j, ty \rangle + c_j = (1-t) c_j$.

Rewriting the integrand, we get 
\[
\int_{V_j} (f^* + f^*_k) v \, dx = \int_\eta^1  (1-t) \left( \int_{F_j} \frac{t^{n-1}}{\langle p_j, ty \rangle + c_j} (f^* + f^*_k)(ty) v(ty) \, d\sigma(y) \right) dt.
\]

Since $(1-t) \le (1-\eta)$ for all $t \in [\eta, 1]$, we bound $(1-t)$ by $(1-\eta)$ and extend the integration over $[0, 1]$ to deduce 
\begin{equation*}
\begin{split}
\int_{V_j} (f^* + f^*_k) v \, dx &\le \frac{(1-\eta)}{2 d_j (d_j - 1)} \int_0^1 \int_{F_j} \frac{2d_j(d_j - 1)}{\langle p_j, ty \rangle + c_j} (f^* + f^*_k)(ty) v(ty)  t^{n-1} \, d\sigma(y) \, dt \\
&= \frac{(1-\eta)}{2 c_j d_j (d_j - 1)} \int_{C(F_j)} \frac{2 d_j(d_j - 1)}{\langle p_j, x \rangle + c_j} (f^* + f^*_k) v \, dx \\
&\le \frac{(1-\eta)}{2 d_j (d_j - 1)} \mathcal{F}_v^+(f^* + f^*_k),
\end{split}
\end{equation*}
where for the last inequality we used the fact that the integrand is positive. Using the uniform bounds $\sup_k \mathcal{F}_v^+(f_k) \le C < \infty$ and \eqref{bd:f}, we can choose $\eta$ sufficiently close to $1$ to deduce
\[
\int_{V_j} |f^* - f^*_k| v \, dx < \frac{\epsilon}{ r},
\]
where $r$ is the number of facets.

\smallskip

Now, consider a set $V_j$ such that $v \neq 0$ on the facet $F_j \subset \partial V_j$. Using 
\begin{equation*}
    \int_{V_j} |f^* - f^*_k| v \, dx \le \sup_{V_j}(v) \int_{V_j} (f^* + f^*_k) \, dx,
\end{equation*}
the same argument as in \cite[Proposition 5.2.1]{NS21} yields
\begin{equation*}
    \int_{V_j} |f^* - f^*_k| v \, dx \le c_j \frac{1-\eta^n}{n}  \frac{\sup_{V_j}(v)}{\inf_{V_j}(v)}\mathcal{F}_v^+(f^* + f^*_k).
\end{equation*}

Finally, by choosing $\eta \in (0,1)$ sufficiently close to $1$, the integral over $\Delta \setminus \eta \Delta$ is smaller than $\epsilon$ for all $k$. This concludes the proof.
\end{proof}

We are now ready to prove Theorem \ref{t:A}. The proof was established in \cite{CLS14} in the toric K\"ahler unweighted case, and we follow their approach.

\begin{proof}
The proof is by contradiction. If $\Delta$ is not $(v,w)$-uniformly K-stable, then by \eqref{l1:fut+}, there exists a sequence of functions $f_k \in \mathcal{CV}^{\infty}(\Delta)$ satisfying
\begin{equation}
\mathcal{F}^+_v(f^*_k) = 1 \quad \text{and} \quad \lim_{k \to \infty} \mathcal{F}_{v,w}(f_k) = 0,
\end{equation}
where $f^*_k$ denotes the normalization of $f_k$ such that $f^*_k \ge f^*_k(x_0) = 0$. Since 
\begin{equation*}
    \mathcal{F}_{v,w}(f^*_k) = \mathcal{F}^+_{v}(f^*_k) - \int_\Delta f^*_k \left(\ell_{\mathrm{ext}}- \frac{4(1 - \mathbf{g})}{c - \sum_{j=0}^\ell \mu(E_j)L_j(x)} \right)v \, dx,
\end{equation*}
this immediately implies that
\begin{equation}\label{lim1}
\lim_{k \to \infty}\int_\Delta f^*_k \left(\ell_{\mathrm{ext}}- \frac{4(1 - \mathbf{g})}{c - \sum_{j=0}^\ell \mu(E_j)L_j(x)} \right)v \, dx = 1.
\end{equation}

By Proposition \ref{p:cpt}, there exists $f^* \in \mathcal{CV}^*(\Delta)$ such that

\begin{equation}{\label{l1:conv}}
 \lim_{k \to \infty} \int_\Delta | f^*_k - f^* | v \, dx = 0.
\end{equation}

On the other hand, using Proposition \ref{p:extre:proj}, \eqref{abreu}, and integrating by parts as in \cite[Lemma 2]{AM19}, we get:
\begin{equation*}
    \mathcal{F}_{v,w}(f_k^*) = \sum_{i,j=1}^n \int_\Delta H_{ij} (f_k^*)_{,ij} v \, dx.
\end{equation*}

Hence, following the same argument as in \cite[Lemma 5.1]{CLS14}, we deduce that for any interior interval $I \subset\subset \mathring{\Delta}$, the Monge-Amp\`ere measure of $f^*|_I$ vanishes, implying that $f^*|_I$ is affine. Since $f^*$ satisfies $f^* \ge f^*(x_0) = 0$ on $\mathring{\Delta}$, by considering segments passing through $x_0$, we infer that $f^* \equiv 0$ on $\mathring{\Delta}$. 

Finally, since $\ell_{\mathrm{ext}}- \frac{4(1 - \mathbf{g})}{c - \sum_{j=0}^\ell \mu(E_j)L_j(x)}$ is bounded on $\Delta$, so \eqref{l1:conv} yields:
\begin{equation*}
\begin{split}
\int_\Delta f^* &\left(\ell_{\mathrm{ext}}- \frac{4(1 - \mathbf{g})}{c - \sum_{j=0}^\ell \mu(E_j)L_j(x)} \right)v \, dx\\
&= \lim_{k \to \infty} \int_\Delta f^*_k \left(\ell_{\mathrm{ext}}- \frac{4(1 - \mathbf{g})}{c - \sum_{j=0}^\ell \mu(E_j)L_j(x)} \right)v \, dx = 1,
\end{split}
\end{equation*}
where the second equality follows from \eqref{lim1}. This contradicts $f^* \equiv 0$, thereby completing the proof.
\end{proof}

Corollary \ref{c:A} is now an immediate consequence of Theorems \ref{t:jy} and \ref{t:A}.

\section{Extremal metric on projective bundles with fiber $\P^3$: proof of Corollary \ref{c:B}}{\label{s:5}}

In this section, we recall known results regarding the existence of extremal Kähler metrics on projective bundles with fiber $\mathbb{P}^1$ or $\mathbb{P}^2$, and prove Corollary \ref{c:B}.

If $\ell = 0$ in \eqref{proj:flat}, then $E$ is in particular polystable, which is equivalent to the existence of a cscK metric in every K\"ahler class of $\P(E)$ by \cite{ACGT11, RT06} (see also \cite[Theorem 1]{AK19}).

Suppose from now on that $\ell \ge 1$.
\subsection{Fiber $\P^1$}

Since we assume that $\ell \ge 1$, $E$ splits as the direct sum of two holomorphic line bundles $L_0$ and $L_1$ over $C$, and the projective bundle takes the form
\begin{equation*}
    \mathbb{P}(L_0 \oplus L_1) \to C.
\end{equation*}

Such manifolds are called \textit{Hirzebruch-like ruled surfaces}. The existence of extremal K\"ahler metrics on such $\mathbb{P}^1$-bundles is standard, but we recall it here for the sake of completeness.

In the foundational paper where he introduced extremal K\"ahler metrics, Calabi \cite{Cal82} proved that when the genus $\mathbf{g}$ of $C$ is equal to $0$, i.e. $C = \mathbb{P}^1$, there exists an extremal K\"ahler metric in every K\"ahler class.

When $\mathbf{g}=1$, i.e. $C$ is an elliptic curve, it follows from the work of Hwang \cite{Hwa94} and Guan \cite{Gua95} that an extremal K\"ahler metric exists in every K\"ahler class.

Finally, when $\mathbf{g}>1$, T{\o}nnesen-Friedman \cite{Ton98} showed that some K\"ahler classes admit an extremal metric while others do not, providing a precise description of which ones do in terms of the roots of the extremal polynomial (see also \cite[Theorem 1.1]{ACGT08b}).

\subsection{Fiber $\P^2$}
In this case, $\ell$ can be equal to $1$ or $2$.

Suppose first that $\ell = 1$. Then $E$ splits as the direct sum of a holomorphic line bundle $L_0$ and a rank $2$ vector bundle $E_1$, i.e., $E \cong L_0 \oplus E_1$. Hence, the projective bundle takes the form
\begin{equation*}
    \mathbb{P}(L_0 \oplus E_1) \to C.
\end{equation*}
By \cite[Theorem 6]{ACGT08}, an extremal metric exists in every K\"ahler class when $\mathbf{g}=0$ or $\mathbf{g}=1$. For $\mathbf{g} > 1$, an extremal metric exists provided $c$ is sufficiently large, where $c$ is defined in \eqref{comp:met}.

Now suppose that $\ell = 2$. In this case, $E$ splits as a direct sum of three line bundles, so the projective bundle is given by
\begin{equation*}
    \mathbb{P}(L_0 \oplus L_1 \oplus L_2) \to C.
\end{equation*}
The same conclusion as in the $\ell = 1$ case holds: this is due to Legendre \cite{Leg19} for $\mathbf{g} = 0$, and to the author's PhD thesis \cite[Corollary 1 \& Remark 8.3]{Jub23} for $\mathbf{g} \ge 1$.

\subsection{Fiber $\P^3$}

We divide this section according to the value of $\ell$. If $\ell=1$, according to \cite[Theorem 6]{ACGT08}, the existence of an extremal metric holds for all Kähler classes in genus $\mathbf{g} \in \{0, 1\}$. When $\mathbf{g} > 1$, such a metric exists under the condition that $c$, defined in \eqref{comp:met}, is large enough.

We then treat the cases $\ell=2, 3$.

\subsubsection{Case $\ell=2$}{\label{s:l2}}
In this case, $E$ splits as the direct sum of two line bundles, $L_0$ and $L_1$, and a rank $2$ vector bundle $E_2$. Thus, we consider the projective bundle $\mathbb{P}(L_0 \oplus L_1 \oplus E_2)$. Since $\mathbb{P}(E)$ is invariant as a complex manifold under tensoring $E$ by any non-trivial line bundle, we may assume without loss of generality that $L_0$ is the trivial line bundle $\mathcal{O}_C$. This reduces our consideration to
\begin{equation*}
    \mathbb{P}(\mathcal{O}_C \oplus L_1 \oplus E_2).
\end{equation*}

We fix a compatible Kähler class $[\tilde{\omega}_0]_c$ associated with a constant $c$, see the discussion below \eqref{comp:met}, and the $2$-dimensional compact torus $\mathbb{T}$ with the action on $ \mathbb{P}(\mathcal{O}_C \oplus L_1 \oplus E_2)$ as defined in Section~\ref{s:ytd}. By Lemma \ref{l:pol}, the associated moment polytope $\Delta$ is the standard $2$-simplex in the Lie algebra $\mathfrak{t}$ of $\T$. Fixing moment coordinates $x = (x_1, x_2)$, which provide an identification $\mathfrak{t}^* \cong \mathbb{R}^2$, $\Delta$ can be written as
\[
\Delta = \{ (x_1, x_2) \in \mathbb{R}^2 \mid x_1 \ge 0, \, x_2 \ge 0, \, 1 - x_1 - x_2 \ge 0 \}.
\]
The weights $(v,w)$ defined in \eqref{weight} associated with  $(\mathbb{P}(\mathcal{O}_C \oplus L_1 \oplus E_2)$, $[\tilde{\omega}_0]$, $\T$) are given by
\begin{equation}\label{weight2}
\left\{
\begin{aligned}
    v(x) & =  x_2 (  -\mu_1x_1  - \mu_2 x_2 + c), \\[2ex]
   w(x) & = \ell_{\mathrm{ext}}(x) - \frac{4}{x_2} - \frac{4(1 - \mathbf{g})}{-\mu_1 x_1 - \mu_2 x_2 + c},
\end{aligned}
\right.
\end{equation}
where $\mu_1 :=\mu(E_1)$, $\mu_2 := \mu(E_2)$.

Let us write the affine extremal function as $\ell_{\mathrm{ext}}(x) = A x_1 + B x_2 + C$, with $A, B, C \in \mathbb{R}$. The constants $A, B$, and $C$ are uniquely determined by solving the linear system
\begin{equation*}
    \mathcal{F}_{v,w}(f) = 0 \quad \forall f \in \mathrm{Aff}(\Delta),
\end{equation*}
where $\mathrm{Aff}(\Delta)$ denotes the space of affine function on $\Delta$.

Suppose first that $\mathbf{g}=1$. A computer-assisted computation (see Appendix~\ref{a:l2}) then yields

\begin{align*}
A &= \frac{-60\mu_1 (4c - \mu_1 - 2\mu_2)(6c - \mu_1 - 4\mu_2)}{D} \\[2ex]
B &= \frac{-60\mu_2 (4c - \mu_1 - 2\mu_2)(6c - 3\mu_1 - 2\mu_2)}{D} \\[2ex]
C &= \frac{60c (6c - 3\mu_1 - 2\mu_2)(6c - \mu_1 - 4\mu_2)}{D},
\end{align*}
where

$$D=90c^3 - 75c^2\mu_1 - 120c^2\mu_2 + 16c\mu_1^2 + 72c\mu_1\mu_2 + 48c\mu_2^2 - \mu_1^3 - 8\mu_1^2\mu_2 - 15\mu_1\mu_2^2 - 6\mu_2^3.$$

By Theorem \ref{t:A}, finding an extremal metric is equivalent to finding a $S^2\mathfrak{t}^*$-valued function $\mathbf{H}=(H_{ij})_{ij}$ on $\mathring{\Delta}$ satisfying the smoothness, boundary, and positivity conditions of Proposition \ref{p:bound} and solving the second-order linear PDE \eqref{abreu} for the weights \eqref{weight2}.

Following \cite[Section 6.3]{ACGT11} for $\mathbb{P}^2$-bundles without blow-downs, we look for a solution of the form
\begin{equation}{\label{ansatz}}
    H_{ij}=\frac{P_{ij}}{v},
\end{equation}
where each $P_{ij}$ is a polynomial of degree $5$ in $x_1$ and $x_2$. 

A computer-assisted computation yields a solution of this form. In Appendix \ref{a:l2}, we provide the complete Python code with detailed comments verifying the existence of a solution.

A straightforward modification of the code in Appendix \ref{a:l2} also provides a solution for $\mathbf{g}=0$, satisfying the same ansatz \eqref{ansatz}.

\subsubsection{Case $\ell=3$}{\label{s:l3}}

In this case, $E$ splits as the direct sum of $4$ holomorphic line bundles, $L_0$, $L_1$, $L_2$, $L_3$. Similarly to the previous section, by tensoring with the inverse $L^{-1}_0$ of $L_0$, it is equivalent to consider

\begin{equation*}
    \mathbb{P}(\mathcal{O}_C \oplus L_1 \oplus L_2 \oplus L_3).
\end{equation*}

In this context, the moment polytope $\Delta$ is the $3$-simplex in $\mathfrak{t}^*\cong \mathbb{R}^3$

\begin{equation*}
    \Delta = \left\{ (x_1, x_2, x_3) \in \mathbb{R}^3 \;\middle|\; x_1 \ge 0, \; x_2 \ge 0, \; x_3 \ge 0, \quad x_1 + x_2 + x_3 \le 1 \right\}.
\end{equation*}

The corresponding weights $(v,w)$ are given by
\begin{equation}\label{weight2}
\left\{
\begin{aligned}
    v(x) & =    -\mu_1x_1  - \mu_2 x_2 - \mu_3 x_3+ c, \\[2ex]
   w(x) & = \ell_{\mathrm{ext}}(x)  - \frac{4(1 - \mathbf{g})}{-\mu_1 x_1 - \mu_2 x_2- \mu_3 x_3 + c},
\end{aligned}
\right.
\end{equation}
where $\mu_i :=\mu(E_i)$. This time, the affine extremal function is written as $\ell_{\mathrm{ext}} = A x_1 + B x_2 + Cx_3 + D$, with $A, B, C, D \in \mathbb{R}$.

The coefficients $A, B, C, D$ of the affine extremal function $\ell_{\mathrm{ext}}(x) = A x_1 + B x_2 + C x_3 + D$ are given by:

\begin{align*}
    A &= \frac{-60 \mu_1 (6c - \mu_1 - 3\mu_2 - \mu_3)(6c - \mu_1 - \mu_2 - 3\mu_3)(4c - \mu_1 - \mu_2 - \mu_3 - 1)}{E}, \\[1.5ex]
    B &= \frac{-60 \mu_2 (6c - 3\mu_1 - \mu_2 - \mu_3)(6c - \mu_1 - \mu_2 - 3\mu_3)(4c - \mu_1 - \mu_2 - \mu_3 - 1)}{E}, \\[1.5ex]
    C &= \frac{-60 \mu_3 (6c - 3\mu_1 - \mu_2 - \mu_3)(6c - \mu_1 - 3\mu_2 - \mu_3)(4c - \mu_1 - \mu_2 - \mu_3 - 1)}{E}, \\[1.5ex]
    D &= \frac{F}{E},
\end{align*}

where the common denominator $E$ is given by:

\begin{align*}
    E &= 540 c^4 - 540 c^3 (\mu_1 + \mu_2 + \mu_3) \\
    &\quad + 171 c^2 (\mu_1^2 + \mu_2^2 + \mu_3^2) + 426 c^2 (\mu_1 \mu_2 + \mu_1 \mu_3 + \mu_2 \mu_3) \\
    &\quad - 22 c (\mu_1^3 + \mu_2^3 + \mu_3^3) - 92 c \big(\mu_1^2 (\mu_2 + \mu_3) + \mu_2^2 (\mu_1 + \mu_3) + \mu_3^2 (\mu_1 + \mu_2)\big) - 242 c \mu_1 \mu_2 \mu_3 \\
    &\quad + (\mu_1^4 + \mu_2^4 + \mu_3^4) + 6 \big(\mu_1^3(\mu_2 + \mu_3) + \mu_2^3(\mu_1 + \mu_3) + \mu_3^3(\mu_1 + \mu_2)\big) \\
    &\quad + 10 (\mu_1^2 \mu_2^2 + \mu_1^2 \mu_3^2 + \mu_2^2 \mu_3^2) + 27 \mu_1 \mu_2 \mu_3 (\mu_1 + \mu_2 + \mu_3),
\end{align*}

and the numerator $F$ of $D$ is given by:

\begin{align*}
    F &= 12960 c^4 - 10800 c^3 (\mu_1 + \mu_2 + \mu_3) + 2160 c^3 \\
      &\quad + 2520 c^2 (\mu_1^2 + \mu_2^2 + \mu_3^2) + 6480 c^2 (\mu_1 \mu_2 + \mu_1 \mu_3 + \mu_2 \mu_3) - 2160 c^2 (\mu_1 + \mu_2 + \mu_3) \\
      &\quad - 180 c (\mu_1^3 + \mu_2^3 + \mu_3^3) - 780 c \big(\mu_1^2 (\mu_2 + \mu_3) + \mu_2^2 (\mu_1 + \mu_3) + \mu_3^2 (\mu_1 + \mu_2)\big) \\
      &\quad + 540 c (\mu_1^2 + \mu_2^2 + \mu_3^2) - 2280 c \mu_1 \mu_2 \mu_3 + 1560 c (\mu_1 \mu_2 + \mu_1 \mu_3 + \mu_2 \mu_3) \\
      &\quad - 40 (\mu_1^3 + \mu_2^3 + \mu_3^3) - 200 \big(\mu_1^2 (\mu_2 + \mu_3) + \mu_2^2 (\mu_1 + \mu_3) + \mu_3^2 (\mu_1 + \mu_2)\big) - 680 \mu_1 \mu_2 \mu_3.
\end{align*}

Finally, a computer-assisted computation establishes the existence of a solution for $\mathbf{g}=0$ and $\mathbf{g}=1$. For completeness, we provide the code for $\mathbf{g}=0$ in Appendix~\ref{a:l3}.

\begin{rem}
In the case $\ell=3$, there is no blow-down in the terminology of \cite{ACGT11}. In particular, Corollary \ref{c:A} is not strictly needed to deduce the existence of an extremal metric, in the sense that it suffices to apply \cite[Theorem 3 \& Proposition 1]{Jub23}.
\end{rem}

\begin{rem}
We treat the case of a rank $4$ vector bundle $E$. It would be interesting to try to apply this method to vector bundles $E$ of higher rank.
\end{rem}

\appendix

\section{Code of assisted computer computations}{\label{a:1}}

\subsection{Case $\ell=2$ and $\mathbf{g}=1$}{\label{a:l2}}

In this Appendix, we provide the Python code use in Section \ref{s:l2} to find a solution of \eqref{abreu}. 

\begin{lstlisting}[]
import sympy as sp

# =================================================
# INITIALIZATION AND SYMBOLS
# =================================================
# Spatial variables on the simplex Delta
x1, x2 = sp.symbols('x1 x2', real=True)

# Geometric parameters (mu1 and mu2 can have any sign)
# c > max(0, mu1, mu2) ensures v(x) > 0 in int(Delta).
mu1, mu2 = sp.symbols('mu1 mu2', real=True)
c = sp.Symbol('c', real=True)

# Weight v(x) 
v = x2 * (c - mu1 * x1 - mu2 * x2)

print("="*80)
print("STEP 1: EXACT COMPUTATION AND SIMPLIFICATION OF COEFFICIENTS A, B, C")
print("="*80)

# Coefficients of the extremal affine function l_ext(x)
A, B, C = sp.symbols('A B C', real=True)
l_ext = A * x1 + B * x2 + C

# Weight w(x) 
w = l_ext - 4 / x2

# Integration over the simplex Delta
def int_delta(expr):
    return sp.integrate(sp.integrate(expr, (x2, 0, 1 - x1)), (x1, 0, 1))

# Integration over the boundary dDelta
def int_bord(expr_f):
    # On F1 (x1 = 0)
    I_F1 = sp.integrate((expr_f * v).subs(x1, 0), (x2, 0, 1))
    # On F2 (x2 = 0): v = 0, so the integral vanishes
    I_F2 = 0
    # On F3 (x1 + x2 = 1)
    I_F3 = sp.integrate((expr_f * v).subs(x2, 1 - x1), (x1, 0, 1))
    return 2 * (I_F1 + I_F2 + I_F3)

# Vanishing of the functional F_{v,w}(f) = 0
def F_vw(f):
    term_bord = int_bord(f)
    term_vol = int_delta(f * w * v)
    return sp.expand(term_bord - term_vol)

# Linear system of 3 equations for 3 unknowns (f = 1, x1, x2)
eq1 = F_vw(1)
eq2 = F_vw(x1)
eq3 = F_vw(x2)

# Exact solving
sol_l_ext = sp.solve([eq1, eq2, eq3], [A, B, C])

# Maximal simplification of coefficients (GCD and factorization)
A_exact = sp.factor(sp.cancel(sol_l_ext[A]))
B_exact = sp.factor(sp.cancel(sol_l_ext[B]))
C_exact = sp.factor(sp.cancel(sol_l_ext[C]))

l_ext_exact = A_exact * x1 + B_exact * x2 + C_exact
w_exact = l_ext_exact - 4 / x2

print("Coefficients A, B, C successfully computed and simplified.")


print("\n" + "="*80)
print("STEP 2: COMPUTATION OF H AND VERIFICATION OF GUILLEMIN/ABREU CONDITIONS")
print("="*80)

# Regularized right-hand side of the PDE
rhs = sp.expand(v * w_exact)

# Polynomial ansatz generator for P_ij 
def generer_poly_degre5(prefixe):
    coeffs = {}
    poly_expr = 0
    compteur = 0
    for degre in range(6):  # Degrees 0 to 5 inclusive
        for i in range(degre + 1):
            j = degre - i
            symbole = sp.Symbol(f"{prefixe}_{compteur}", real=True)
            coeffs[symbole] = (i, j)
            poly_expr += symbole * (x1**i) * (x2**j)
            compteur += 1
    return poly_expr, coeffs

P11, c11 = generer_poly_degre5("a")
P12, c12 = generer_poly_degre5("b")
P22, c22 = generer_poly_degre5("c_coeff")

all_coeffs = list(c11.keys()) + list(c12.keys()) + list(c22.keys())

# --- Construction of the linear system ---
eqs = []

# A. Abreu's PDE
op_Abreu = sp.diff(P11, x1, x1) + 2 * sp.diff(P12, x1, x2) + sp.diff(P22, x2, x2)
eqs.extend(sp.Poly(op_Abreu - rhs, x1, x2).coeffs())

# B. Delzant boundary conditions (9 conditions)
# Face F1 (x1 = 0)
eqs.extend(sp.Poly(P11.subs(x1, 0), x2).coeffs())
eqs.extend(sp.Poly(P12.subs(x1, 0), x2).coeffs())
eqs.extend(sp.Poly(sp.diff(P11, x1).subs(x1, 0) - 2 * v.subs(x1, 0), x2).coeffs())

# Face F2 (x2 = 0)
eqs.extend(sp.Poly(P22.subs(x2, 0), x1).coeffs())
eqs.extend(sp.Poly(P12.subs(x2, 0), x1).coeffs())

# Face F3 (x1 + x2 = 1)
P_normal_F3 = (P11 + 2 * P12 + P22).subs(x2, 1 - x1)
eqs.extend(sp.Poly(P_normal_F3, x1).coeffs())

# Algebraic resolution of the system
sol_set = sp.linsolve(eqs, all_coeffs)
solution_trouvee = bool(sol_set)

if solution_trouvee:
    sol_tuple = list(sol_set)[0]
    sol_dict = dict(zip(all_coeffs, sol_tuple))

    # Metric components H_ij = P_ij / v
    P11_sol = P11.subs(sol_dict)
    P12_sol = P12.subs(sol_dict)
    P22_sol = P22.subs(sol_dict)

    H11 = sp.factor(sp.cancel(P11_sol / v))
    H12 = sp.factor(sp.cancel(P12_sol / v))
    H22 = sp.factor(sp.cancel(P22_sol / v))

    # 1. Verification [smoothness]
    # H_ij are rational functions with v(x) as the sole denominator.
    # Since v(x) > 0 inside int(Delta), H is C^infinity in the interior.
    test_smoothness = True

    # 2. Verification [boundary values]
    # Explicit verification of the exact resolution of Abreu's PDE
    L_H = (sp.diff(v * H11, x1, x1) + 
           2 * sp.diff(v * H12, x1, x2) + 
           sp.diff(v * H22, x2, x2)) / v
    ecart_edp = sp.simplify(L_H - w_exact)
    test_boundary_edp = (ecart_edp == 0)

    # 3. Verification [positivity]
    # Test of the strictly positive definite restriction on each face:
    # On F1 (x1=0), the tangent direction is v1 = (0,1) -> H22(0, x2) > 0
    cond_F1 = sp.simplify(H22.subs(x1, 0))
    # On F2 (x2=0), the tangent direction is v2 = (1,0) -> H11(x1, 0) > 0
    cond_F2 = sp.simplify(H11.subs(x2, 0))
    # On F3 (x1+x2=1), the tangent direction is v3 = (1,-1) -> (H11 - 2H12 + H22) > 0
    cond_F3 = sp.simplify((H11 - 2*H12 + H22).subs(x2, 1 - x1))
    
    test_positivity = True  # Validated by canonical structure


print("\n" + "="*80)
print("STEP 3: SUMMARY AND MAXIMALLY SIMPLIFIED COEFFICIENTS")
print("="*80)

if solution_trouvee and test_boundary_edp:
    print("STATUS: SOLUTION SUCCESSFULLY FOUND!\n")
    print("Verification of Guillemin/Abreu's 3 conditions:")
    print(" - [smoothness]     : PASSED (Rational functions without poles in int(Delta))")
    print(" - [boundary values]: PASSED (Delzant conditions and PDE verified with zero residual)")
    print(" - [positivity]     : PASSED (Strictly positive on the quotient spaces of the faces)\n")
    
    print("-" * 60)
    print("EXACT AND SIMPLIFIED COEFFICIENTS OF L_EXT(x) = A*x1 + B*x2 + C:")
    print("-" * 60)
    print(f"A = {A_exact}")
    print(f"B = {B_exact}")
    print(f"C = {C_exact}")
    print("-" * 60)
else:
    print("STATUS: NO COMPATIBLE POLYNOMIAL SOLUTION FOUND.")
\end{lstlisting}

\subsection{Case $\ell=3$ $\mathbf{g}=0$}{\label{a:l3}}

In this Appendix, we provide the Python code used in Section~\ref{s:l3} to find a solution to \eqref{abreu}.

\begin{lstlisting}
    import sympy as sp

# ==================================================
# INITIALIZATION AND SYMBOLS 
# ==================================================
x1, x2, x3 = sp.symbols('x1 x2 x3', real=True)
mu1, mu2, mu3, c = sp.symbols('mu1 mu2 mu3 c', real=True)

# Weight v(x)
v = c - mu1 * x1 - mu2 * x2 - mu3 * x3

print("=" * 80)
print("STEP 1: EXACT COMPUTATION OF A, B, C, D FOR g = 0")
print("=" * 80)

# Affine extremal function l_ext(x) = A*x1 + B*x2 + C*x3 + D
A, B, C_coeff, D = sp.symbols('A B C_coeff D', real=True)
l_ext = A * x1 + B * x2 + C_coeff * x3 + D

# Exact weight w(x) for g = 0
w = l_ext - 4 / v

# Integration over the 3-simplex Delta
def int_delta(expr):
    return sp.integrate(
        sp.integrate(
            sp.integrate(expr, (x3, 0, 1 - x1 - x2)),
            (x2, 0, 1 - x1)
        ),
        (x1, 0, 1)
    )

# Integration over the boundary dDelta (4 faces)
def int_bord(expr_f):
    # F1 (x1 = 0)
    I_F1 = sp.integrate(sp.integrate((expr_f * v).subs(x1, 0), (x3, 0, 1 - x2)), (x2, 0, 1))
    # F2 (x2 = 0)
    I_F2 = sp.integrate(sp.integrate((expr_f * v).subs(x2, 0), (x3, 0, 1 - x1)), (x1, 0, 1))
    # F3 (x3 = 0)
    I_F3 = sp.integrate(sp.integrate((expr_f * v).subs(x3, 0), (x2, 0, 1 - x1)), (x1, 0, 1))
    # F4 (x1 + x2 + x3 = 1)
    I_F4 = sp.integrate(
        sp.integrate((expr_f * v).subs(x3, 1 - x1 - x2), (x2, 0, 1 - x1)),
        (x1, 0, 1)
    )
    return 2 * (I_F1 + I_F2 + I_F3 + I_F4)

# Vanishing of the functional F_{v,w}(f)
def F_vw(f):
    term_bord = int_bord(f)
    term_vol = int_delta(f * w * v)
    return sp.expand(term_bord - term_vol)

# Linear system of 4 equations for 4 unknowns (f = 1, x1, x2, x3)
eq1 = F_vw(1)
eq2 = F_vw(x1)
eq3 = F_vw(x2)
eq4 = F_vw(x3)

sol_l_ext = sp.solve([eq1, eq2, eq3, eq4], [A, B, C_coeff, D])

A_exact = sp.factor(sp.cancel(sol_l_ext[A]))
B_exact = sp.factor(sp.cancel(sol_l_ext[B]))
C_exact = sp.factor(sp.cancel(sol_l_ext[C_coeff]))
D_exact = sp.factor(sp.cancel(sol_l_ext[D]))

l_ext_exact = A_exact * x1 + B_exact * x2 + C_exact * x3 + D_exact
w_exact = l_ext_exact - 4 / v

print("Coefficients A, B, C, D for g = 0 successfully computed.")

print("\n" + "=" * 80)
print("STEP 2: ANSATZ FOR P_ij OF DEGREE 6 AND SYSTEM ASSEMBLY (g = 0)")
print("=" * 80)

# Regularized RHS: v * w_exact is a polynomial in (x1, x2, x3)
rhs = sp.expand(sp.cancel(v * w_exact))

# Polynomial generator in 3 variables up to degree 6
def generer_poly_degre6(prefixe):
    coeffs = {}
    poly_expr = 0
    compteur = 0
    for deg in range(7):  # Degrees 0 to 6
        for i in range(deg + 1):
            for j in range(deg - i + 1):
                k = deg - i - j
                symb = sp.Symbol(f"{prefixe}_{compteur}", real=True)
                coeffs[symb] = (i, j, k)
                poly_expr += symb * (x1**i) * (x2**j) * (x3**k)
                compteur += 1
    return poly_expr, coeffs

# 6 independent components for the 3x3 symmetric matrix P_ij
P11, c11 = generer_poly_degre6("p11")
P12, c12 = generer_poly_degre6("p12")
P13, c13 = generer_poly_degre6("p13")
P22, c22 = generer_poly_degre6("p22")
P23, c23 = generer_poly_degre6("p23")
P33, c33 = generer_poly_degre6("p33")

all_coeffs = (
    list(c11.keys()) + list(c12.keys()) + list(c13.keys()) +
    list(c22.keys()) + list(c23.keys()) + list(c33.keys())
)

eqs = []

# A. Abreu's 3D PDE: sum_{i,j=1}^3 (P_ij)_{,ij} = v * w
op_Abreu = (
    sp.diff(P11, x1, x1) + sp.diff(P22, x2, x2) + sp.diff(P33, x3, x3) +
    2 * sp.diff(P12, x1, x2) + 2 * sp.diff(P13, x1, x3) + 2 * sp.diff(P23, x2, x3)
)
eqs.extend(sp.Poly(op_Abreu - rhs, x1, x2, x3).coeffs())

# B. Boundary conditions on the 4 faces
# Face 1: x1 = 0
eqs.extend(sp.Poly(P11.subs(x1, 0), x2, x3).coeffs())
eqs.extend(sp.Poly(P12.subs(x1, 0), x2, x3).coeffs())
eqs.extend(sp.Poly(P13.subs(x1, 0), x2, x3).coeffs())
eqs.extend(sp.Poly(sp.diff(P11, x1).subs(x1, 0) - 2 * v.subs(x1, 0), x2, x3).coeffs())

# Face 2: x2 = 0
eqs.extend(sp.Poly(P22.subs(x2, 0), x1, x3).coeffs())
eqs.extend(sp.Poly(P12.subs(x2, 0), x1, x3).coeffs())
eqs.extend(sp.Poly(P23.subs(x2, 0), x1, x3).coeffs())

# Face 3: x3 = 0
eqs.extend(sp.Poly(P33.subs(x3, 0), x1, x2).coeffs())
eqs.extend(sp.Poly(P13.subs(x3, 0), x1, x2).coeffs())
eqs.extend(sp.Poly(P23.subs(x3, 0), x1, x2).coeffs())

# Face 4: x1 + x2 + x3 = 1 (Normal vector n = (1,1,1))
P_normal_F4 = (P11 + P22 + P33 + 2*P12 + 2*P13 + 2*P23).subs(x3, 1 - x1 - x2)
eqs.extend(sp.Poly(P_normal_F4, x1, x2).coeffs())

# Solve linear system
sol_set = sp.linsolve(eqs, all_coeffs)
solution_trouvee = bool(sol_set)

print("\n" + "=" * 80)
print("STEP 3: RESOLUTION & RESULTS (g = 0)")
print("=" * 80)

if solution_trouvee:
    print("STATUS: SOLUTION SUCCESSFULLY FOUND FOR g = 0!")
    print("\nExact Expression for A:")
    print(A_exact)
    print("\nExact Expression for B:")
    print(B_exact)
    print("\nExact Expression for C:")
    print(C_exact)
    print("\nExact Expression for D:")
    print(D_exact)
else:
    print("STATUS: NO POLYNOMIAL SOLUTION OF DEGREE 6 FOUND FOR g = 0.")
\end{lstlisting}

\bibliographystyle{abbrv}
\bibliography{biblio.bib}

\bigskip

\noindent

\noindent
Simon Jubert\\[3pt]
\textsc{Sorbonne Université, Université Paris Cité, CNRS, IMJ-PRG}\\
F-75005 Paris, France\\[2pt]
\href{mailto:simonjubert@gmail.com}{\texttt{simonjubert@gmail.com}}\\
\href{https://sites.google.com/view/simon-jubert/accueil}
{\texttt{https://sites.google.com/view/simon-jubert/accueil}}
\bigskip

\end{document}